\documentclass[12pt]{amsart}
\usepackage{amsfonts}
\usepackage{amsmath}
\usepackage{amssymb}
\usepackage[margin=1.2in]{geometry}
\newcommand{\R}{\mathbb{R}}

\newtheorem{theorem}{Theorem}[section]
\newtheorem{proposition}[theorem]{Proposition}
\newtheorem{corollary}[theorem]{Corollary}
\newtheorem{lemma}[theorem]{Lemma}

\theoremstyle{definition}
\newtheorem{definition}[theorem]{Definition}

\newtheorem{remark}[theorem]{Remark}

\numberwithin{equation}{section}
\numberwithin{theorem}{section}

\numberwithin{equation}{section}

\begin{document}
\title[Generalization of the Wiener-Ikehara theorem]{Generalization of the Wiener-Ikehara theorem}

\author[G. Debruyne]{Gregory Debruyne}
\thanks{G. Debruyne gratefully acknowledges support by Ghent University, through a BOF Ph.D. grant.}
\address{G. Debruyne\\ Department of Mathematics\\ Ghent University\\ Krijgslaan 281\\ B 9000 Ghent\\ Belgium}
\email{gregory.debruyne@UGent.be}
\author[J. Vindas]{Jasson Vindas}
\thanks{The work of J. Vindas was supported by the Research Foundation--Flanders, through the FWO-grant number 1520515N}
\address{J. Vindas\\ Department of Mathematics\\ Ghent University\\ Krijgslaan 281\\ B 9000 Ghent\\ Belgium}
\email{jasson.vindas@UGent.be}

\begin{abstract} We study the Wiener-Ikehara theorem under the so-called log-linearly slowly decreasing condition. Moreover, we clarify the connection between two different hypotheses on the Laplace transform occurring in  exact forms of the Wiener-Ikehara theorem, that is, in ``if and only if'' versions of this theorem.
\end{abstract}

\maketitle

\section{Introduction}

The Wiener-Ikehara theorem plays a central role in Tauberian theory \cite{korevaarbook}. Since its publication \cite{ikehara,Wienerbook}, there have been numerous applications and generalizations of this theorem, see, e.g., \cite{Aramaki1996,Debruyne-VindasComplexTauberians,delange1954,grahamvaaler,korevaar2005,revesz-roton, zhang2014}. 

Recently, Zhang has relaxed the non-decreasing Tauberian condition in the Wiener-Ikehara theorem to so-called log-linear slow decrease. Following Zhang, we shall call a function $f$ linearly slowly decreasing if for each $\varepsilon>0$ there is $a>1$ such that 
$$
\liminf_{x\to\infty} \inf_{y\in [x,ax]} \frac{f(y)-f(x)}{x}\geq -\varepsilon,
$$
and we call a function $S$ \emph{log-linearly slowly decreasing} if $S(\log x) $ is linearly slowly decreasing, i.e., if for each $\varepsilon > 0$ there exist $\delta > 0$ and $x_{0}$ such that
\begin{equation}
\label{defloglsd}
 \frac{S(x+h) - S(x)}{e^{x}} \geq -\varepsilon,  \quad \mbox{for } 0 \leq h \leq \delta \mbox{ and } x\geq x_{0}.
\end{equation}
Using the latter condition, Zhang was able to obtain an exact form of the Wiener-Ikehara theorem. His theorem\footnote{W.-B. Zhang communicated Theorem \ref{thzhang} in his talk \emph{Exact Wiener-Ikehara theorems}, presented at the Number Theory Seminar of the University of Illinois at Urbana-Champaign on July 5, 2016.} reads as follows,

\begin{theorem} \label{thzhang} Let $S\in L^{1}_{loc}[0,\infty)$ be log-linearly slowly decreasing. Assume that
\begin{equation} \label{conditionzhang1}
 \mathcal{L}\{S;s\} = \int^{\infty}_{0} e^{-sx}S(x) \mathrm{d}x  \ \ \ \text{is absolutely convergent for } \Re e \: s > 1 
\end{equation}
and that there is a constant $a$ for which 
$$
G(s)= \mathcal{L}\{S;s\} - \frac{a}{s-1}
$$
satisfies: There is $\lambda_{0}>0$ such that for each $\lambda\geq \lambda_0$
 \begin{equation}
 \label{conditionzhang2}
 I_{\lambda}(h)=\lim_{\sigma \rightarrow 1^{+}} \int^{\lambda}_{-\lambda} G(\sigma + it) e^{iht} \left(1 - \frac{\left|t\right|}{\lambda}\right)\mathrm{d}t 
 \end{equation}
exists for all sufficiently large $h>h_{\lambda}$ and 
\begin{equation} \label{conditionzhang3}
 \lim_{h \rightarrow \infty} I_{\lambda}(h)= 0.
 \end{equation}
Then,
\begin{equation} \label{conclusionzhang}
 S(x) \sim ae^{x}.
\end{equation}
\end{theorem}
Theorem \ref{thzhang} is exact in the sense that if (\ref{conclusionzhang}) holds, then $S$ is log-linearly slowly decreasing and (\ref{conditionzhang1})--(\ref{conditionzhang3}) hold as well. Note that the hypotheses (\ref{conditionzhang2}) and (\ref{conditionzhang3}) in Zhang's result cover as particular instances the cases when $\mathcal{L}\{S;s\} - a/(s-1)$ has analytic or even $L^{1}_{loc}$-extension to $\Re e\: s=1$, as follows from the Riemann-Lebesgue lemma.

About a decade ago, Korevaar \cite{korevaar2005} also obtained an exact form of the Wiener-Ikehara theorem for non-decreasing functions. His exact hypothesis on the Laplace transform was the so-called local pseudofunction boundary behavior. The authors have recently established \cite{Debruyne-VindasComplexTauberians} local pseudofunction behavior as a minimal boundary assumption in other complex Tauberian theorems for Laplace transforms.  It should be pointed out that Tauberian theorems with mild boundary hypotheses have relevant applications in the theory of Beurling generalized numbers (cf.~\cite{d-vPNTequiv2016,diamond-zhang2013,diamond-zhangbook,s-v,zhang2014}); in fact, in that setting one must  work with zeta functions whose boundary values typically display very low regularity properties.

In this article we show that local pseudofunction boundary behavior is also able to deliver an exact form of the Wiener-Ikehara theorem if one works with log-linear slow decrease. Moreover, we clarify the connection between local pseudofunction boundary behavior and the exact conditions of Zhang, giving a form of the Wiener-Ikehara theorem that contains both versions (Theorem \ref{thol1}). 

We thank H. G. Diamond and W.-B. Zhang  for useful discussions on the subject.

\section{Pseudofunctions and pseudomeasures}
\label{section preli}
We present in this section some background material on pseudofunctions and pseudomeasures. 

We begin with Fourier transforms, which we shall interpret in the distributional sense. The standard Schwartz test function spaces of compactly supported smooth functions (on an open subset $U\subseteq \mathbb{R}$)  and rapidly decreasing functions are denoted by $\mathcal{D}(U)$ and $\mathcal{S}(\mathbb{R})$, while $\mathcal{D}'(U)$ and $\mathcal{S}'(\mathbb{R})$ stand for their topological duals, the spaces of distributions and tempered distributions. The Fourier transform, normalized as $\hat{\varphi}(t)=\mathcal{F}\{\varphi;t\}=\int_{-\infty}^{\infty}e^{-itx}\varphi(x)\:\mathrm{d}x,$ is a topological automorphism on the Schwartz space $\mathcal{S}(\mathbb{R})$. One can then extend it to $\mathcal{S}'(\mathbb{R})$ via duality, namely, the Fourier transform of $f\in\mathcal{S}'(\mathbb{R})$ is the tempered distribution $\hat{f}\in \mathcal{S}'(\mathbb{R})$ determined by $\langle \hat{f}(t),\varphi(t)\rangle=\langle f(x),\hat{\varphi}(x)\rangle$, for each test function $\varphi\in\mathcal{S}(\mathbb{R})$. As usual, locally integrable functions are regarded as distributions via $\langle f(x),\varphi(x)\rangle=\int_{-\infty}^{\infty}f(x)\varphi(x)\mathrm{d}x$. Note that if $f\in\mathcal{S}'(\mathbb{R})$ has support in $[0,\infty)$, its Laplace transform 
$\mathcal{L}\left\{f;s\right\}=\left\langle f(u),e^{-su}\right\rangle$ is well-defined, analytic on $\Re e\:s>0$, and one has $\lim_{\sigma\to0^{+}}\mathcal{L}\left\{f;\sigma+it\right\}=\hat{f}(t)$, in the distributional sense.
See the textbooks \cite{bremermann,vladimirov} for further details on distribution theory.

Pseudofunctions and pseudomeasures are special kinds of  Schwartz distributions that arise in harmonic analysis \cite{benedettobook,kahane-salem} and are defined via Fourier transform. A tempered distribution $f\in\mathcal{S}'(\mathbb{R})$ is called a (global) pseudomeasure if $\hat{f} \in L^{\infty}(\mathbb{R})$. If we additionally have $\lim_{|x|\to\infty}\hat{f}(x)=0$, we call $f$ a (global) pseudofunction. We denote the spaces of pseudofunctions and pseudomeasures by $PF(\mathbb{R})$ and $PM(\mathbb{R})$, respectively. 

We say that a distribution $g$ is a  pseudofunction (pseudomeasure) at $t_0\in\mathbb{R}$ if the point possesses an open neighborhood where $g$ coincides with a pseudofunction (pseudomeasure). We then say that $g\in\mathcal{D}'(U)$ is a \emph{local} pseudofunction (local pseudomeasure) on an open set $U\subseteq\mathbb{R}$ if $g$ is a pseudofunction (pseudomeasure) at every $t_0\in U$; we write $g\in PF_{loc}(U)$ ($g\in PM_{loc}(U)$). Using a partition of the unity, one easily checks that  $g\in PF_{loc}(U)$ if and only if $\varphi g \in PF(\mathbb{R})$ for each $\varphi\in \mathcal{D}(U)$, or, which amounts to the same, it satisfies \cite{korevaar2005}
\begin{equation}
\label{eqRL}
\left\langle g(t),e^{iht}\varphi(t)\right\rangle=o(1),
\end{equation}
as $|h|\to\infty,$ 
for each $\varphi\in \mathcal{D}(U)$. The property (\ref{eqRL}) can be regarded as a generalized Riemann-Lebesgue lemma. In particular, $L^{1}_{loc}(U)\subset PF_{loc}(U)$. Likewise, if we replace $o(1)$ by $O(1)$ in (\ref{eqRL}), namely,
\begin{equation}
\label{eqRL2}
\left\langle g(t),e^{iht}\varphi(t)\right\rangle=O(1),
\end{equation}
as $|h|\to\infty$, we obtain a characterization of local pseudomeasures. Hence, any Radon measure on $U$ is an instance of a local pseudomeasure.  We mention that smooth functions are multipliers for local pseudofunctions and pseudomeasures, as follows from \eqref{eqRL} and \eqref{eqRL2}.

Let $G(s)$ be analytic on the half-plane $\Re e\:s>\alpha$. We say that $G$ has local pseudofunction (local pseudomeasure) boundary behavior on the boundary open set $\alpha+iU$ if there is $g\in PF_{loc}(U)$ ($g\in PM_{loc}(U)$) such that
\begin{equation}
\label{bveq}
\lim_{\sigma\to\alpha^{+}}\int_{-\infty}^{\infty}G(\sigma+it)\varphi(t)\mathrm{d}t=\left\langle g(t),\varphi(t)\right\rangle\ , \quad \mbox{for each } \varphi\in\mathcal{D}(U).
\end{equation}
The meaning of having pseudofunction (pseudomeasure) boundary behavior at a boundary point $\alpha+it_{0}$ should be clear. We emphasize that $L^1_{loc}$, continuous, or analytic extension are very special cases of local pseudofunction boundary behavior. Interestingly, if $g\in\mathcal{D}'(U)$ is the distributional boundary value of an analytic function, just having 
\eqref{eqRL}  (\eqref{eqRL2}, resp.) as $h\to\infty$ suffices to conclude that $g\in PF_{loc}(U)$ ($g\in PM_{loc}(U)$), as shown by the following proposition.

\begin{proposition} \label{W-Iprop1} Suppose that $g\in\mathcal{D}'(U)$ is the boundary distribution on $\alpha+iU$ of an analytic function $G$ on the half-plane $\Re e\:s>\alpha$, that is, that \eqref{bveq} holds for every test function $\varphi\in\mathcal{D}(U)$. Then, for each $\varphi\in\mathcal{D}(U)$ and $n\in\mathbb{N}$,
\[
\left\langle g(t),e^{iht}\varphi(t)\right\rangle=O\left(\frac{1}{|h|^{n}}\right), \quad h\to-\infty.
\]
In particular, $g$ is a local pseudofunction (local pseudomeasure) on $U$ if and only if \eqref{eqRL} (\eqref{eqRL2}, resp.) holds as $h\to\infty$ for each $\varphi\in\mathcal{D}(U)$.
\end{proposition}
\begin{proof} Fix $\varphi\in \mathcal{D}(U)$ and let $V$ be an open neighborhood of $\operatorname*{supp}\varphi$ with compact closure in $U$.  Pick a distribution $f\in\mathcal{S}'(\mathbb{R})$ such that $\hat{f}$ has compact support and $\hat{f}=g$ on $V$. The Paley-Wiener-Schwartz theorem tells us that $f$ is an entire function with at most polynomial growth on the real axis, so, find $m>0$ such that $f(x)=O (|x|^{m})$, $|x|\to\infty$. Let $f_{\pm}(x)=f(x)H(\pm x)$, where $H$ is the Heaviside function, i.e., the characteristic function of the interval $[0,\infty)$. Observe that \cite{bremermann} $\hat{f}_{\pm}(t)=\lim_{\sigma\to0^{+}} \mathcal{L}\{f_{\pm}; \pm\sigma+it\}$, where the limit is taken in $\mathcal{S}'(\mathbb{R})$. We also have $g=\hat{f}_{-}+\hat{f}_{+}$ on $V$. Consider the analytic function, defined off the imaginary axis,
\[
F(s)=\begin{cases}
G(s+\alpha)- \mathcal{L}\{f_{+}; s\} & \quad \mbox{if }\Re e\: s>0,
\\
\mathcal{L}\{f_{-}; s\} & \quad \mbox{if }\Re e\: s<0.
\end{cases}
\]
The function $F$ has zero distributional jump across the subset $iV$ of the imaginary axis, namely,
\[
\lim_{\sigma\to 0^{+}} F(\sigma +it)- F(-\sigma+it)= 0 \quad \mbox{in }\mathcal{D}'(V).
\]
The edge-of-the-wedge theorem \cite[Thm. B]{rudin1971} gives that $F$ has analytic continuation through $iV$. We then conclude that $\hat{f}_{-}$ must be a real analytic function on $V$. Integration by parts then yields
$$
\left\langle \hat{f}_{-}(t),e^{iht}\varphi(t)\right\rangle= \int_{-\infty}^{\infty} \hat{f}_{-}(t)\varphi(t)e^{iht}\:\mathrm{d}t=O_{n}\left(\frac{1}{|h|^{n}}\right), \quad |h|\to\infty.
$$
On the other hand, as $h\to-\infty$,
\begin{align*}
 \left\langle \hat{f}_{+}(t),e^{iht}\varphi(t)\right\rangle&= \left\langle f_{+}(x),\hat{\varphi}(x-h)\right\rangle=\int_{0}^{\infty} f(x)\hat{\varphi}(x+|h|)\:\mathrm{d}x 
\\
&
\ll_{n,m} \int_{0}^{\infty}\frac{ (x+1)^{m}}{(x+|h|)^{n+m+1}}\:\mathrm{d}x\leq  \frac{1}{|h|^{n}}\int_{0}^{\infty}\frac{ \mathrm{d}u}{(u+1)^{n+1}},
\end{align*}
because $\hat{\varphi}$ is rapidly decreasing.
\end{proof}
 
\section{Generalizations of the Wiener-Ikehara theorem}

We begin our investigation with a boundedness result. We call a function $S$ \emph{log-linearly boundedly decreasing} if there is $\delta>0$ such that
$$
\liminf_{x\to\infty} \inf_{h\in [0,\delta]} \frac{S(x+h)-S(x)}{e^{x}}> -\infty,
$$
that is, if there are $\delta,x_0,M > 0$ such that
\begin{equation}
\label{deflogbsd}
 S(x+h) - S(x) \geq -Me^{x}, \quad \mbox{for } 0 \leq h \leq \delta \mbox{ and } x\geq x_{0}.
\end{equation}
Functions defined on $[0,\infty)$ are always tacitly extended to $(-\infty,0)$ as $0$ for $x<0$.
\begin{proposition} \label{thboundedness} Let $S\in L^{1}_{loc}[0,\infty)$. Then, 
 \begin{equation}
 \label{W-Ieq1}
 S(x) =O(e^{x}), \quad x\to\infty,
\end{equation}
if and only if $S$ is log-linearly boundedly decreasing and its Laplace transform
\begin{equation}\label{LaplaceConvergenceeq}
 \mathcal{L}\{S;s\} = \int^{\infty}_{0} e^{-sx} S(x) \mathrm{d}x  \quad \mbox{converges for } \Re e \: s > 1
\end{equation}
and admits pseudomeasure boundary behavior at the point $s = 1$.

\end{proposition}

\begin{proof} Suppose \eqref{W-Ieq1}  holds. It is obvious that $S$ must be log-linearly boundedly decreasing and that \eqref{LaplaceConvergenceeq} is convergent for $\Re e\:s>1$. Set $\Delta(x) = e^{-x}S(x)$ and decompose it as $\Delta=\Delta_{1}+\Delta_{2}$, where $\Delta_{2}\in L^{\infty}(\mathbb{R})$ and $\Delta_{1}$ is compactly supported. The boundary value of \eqref{LaplaceConvergenceeq} on $\Re e\: s=1$ is the Fourier transform of $\Delta$, that is, the distribution $\hat{\Delta}_{1}+\hat{\Delta}_2$. By definition $\hat{\Delta}_{2}\in PM(\mathbb{R})$, while $\hat{\Delta}_{1}\in C^{\infty}(\mathbb{R}) \subset PF_{loc}(\mathbb{R})$ because it is in fact the restriction of an entire function to the real line. So, actually $\hat{\Delta}\in PM_{loc}(\mathbb{R})$.

Let us now prove that the conditions are sufficient for \eqref{W-Ieq1}. Since changing a function on a finite interval does not violate the local pseudomeasure behavior of the Laplace transform, we may assume that \eqref{deflogbsd} holds for all $x\geq 0$. Iterating the inequality \eqref{deflogbsd}, one finds that there is $C$ such that 
\begin{equation} \label{eqlsditerated}
S(u) - S(y) \geq -Ce^{u} \quad \mbox{for all } u \geq y\geq 0.
\end{equation}
We may thus assume without loss of generality that $S$ is positive. In fact, if necessary, one may replace $S$ by $\tilde{S}(u) = S(u) + S(0) + Ce^{u}$, whose Laplace transform also admits local pseudomeasure boundary behavior at $s = 1$.

We set again $\Delta(x) = e^{-x}S(x)$, its Laplace transform is $ \mathcal{L}\{S;s+1\}$, so that $\mathcal{L}\{\Delta;s\}$ has pseudomeasure boundary behavior at $s = 0$. There are then a sufficiently small $\lambda>0$ and a local pseudomeasure $g$ on $(-\lambda,\lambda)$ such that $\lim_{\sigma\to 0^{+}}\mathcal{L}\{\Delta;\sigma+it\}=g(t)$ in $\mathcal{D}'(-\lambda,\lambda)$.  Let $\varphi$ be an arbitrary (non-identically zero) smooth function with support in $(-\lambda,\lambda)$ such that its Fourier transform $\hat{\varphi}$ is non-negative.
By the monotone convergence theorem and the equality $\mathcal{L}\{\Delta; \sigma + it\} = \mathcal{F}\{\Delta(x) e^{-\sigma x};t\}$ in $\mathcal{S}'(\R)$,
\begin{align*}
\int_{0}^{\infty} \Delta(x) \hat{\varphi}(x-h)\mathrm{d}x & =  \lim_{\sigma \rightarrow 0^{+}} \int^{\infty}_{0} \Delta(x)e^{-\sigma x} \hat{\varphi}(x-h)  \mathrm{d}x\\
& = \lim_{\sigma \rightarrow 0^{+}}\int_{-\infty}^{\infty}\mathcal{L}\{\Delta; \sigma + it\} e^{iht}\varphi(t)  \mathrm{d}t
\\
&
= \left\langle g(t),e^{iht}\varphi(t)\right\rangle =O(1), \quad \mbox{as } h \rightarrow \infty.
\end{align*}
Set now $B = \int^{\infty}_{0} e^{-x}\hat{\varphi}(x)\mathrm{d}x>0$. Appealing to (\ref{eqlsditerated}) once again, we obtain
\begin{align*}
 e^{-h}S(h) & = \frac{1}{B} \int^{\infty}_{0} e^{-x-h}S(h)\hat{\varphi}(x)\mathrm{d}x 
 \\
 &
 \leq \frac{1}{B}\int^{\infty}_{0}e^{-x-h}S(x+h)\hat{\varphi}(x) \mathrm{d}x + \frac{C}{B} \int^{\infty}_{0}\hat{\varphi}(x) \mathrm{d}x \\
& \leq \frac{1}{B}\int^{\infty}_{0} \Delta(x)\hat{\varphi}(x-h)\mathrm{d}x + \frac{C}{B} \int^{\infty}_{0}\hat{\varphi}(x) \mathrm{d}x = O(1). 
\end{align*}
\end{proof}

If one reads the above proof carefully, one realizes that we do not have to ask the existence of $\lambda > 0$ such that 
\begin{equation*}
 \left\langle g(t), e^{iht} \varphi(t)\right\rangle = O(1), \quad h\to\infty, \quad  \mbox{for all }\varphi \in \mathcal{D}(-\lambda,\lambda),
\end{equation*}
where $g$ is as in the proof of Proposition \ref{thboundedness}. Indeed, one only needs one appropriate test function in this relation. To generalize Proposition \ref{thboundedness}, we introduce the ensuing terminology. The Wiener algebra is $A(\mathbb{R})= \mathcal{F}(L^{1}(\mathbb{R}))$. We write $A_{c}(\mathbb{R})$ for the subspace of  $A(\mathbb{R})$ consisting of compactly supported functions.

\begin{definition}
\label{W-Idef1} 
An analytic function $G(s)$ on the half-plane $\Re e \: s > \alpha$ is said to have pseudomeasure boundary behavior (pseudofunction boundary behavior) on $\Re e \: s = \alpha$ with respect to  $\varphi \in A_{c}(\mathbb{R})$ if there is $N>0$ such that
\[
I_{\varphi}(h)= \lim_{\sigma \rightarrow \alpha^{+}} \int^{\infty}_{-\infty} G(\sigma + it) e^{iht}\varphi(t) \mathrm{d}t
\]
exists for every $h\geq N$ and 
$I_{\varphi}(h)= O(1)$ ($I_{\varphi}(h)=o(1)$, resp.) as $h \rightarrow \infty$.
\end{definition}

Let us check that the notions from Definition \ref{W-Idef1} generalize those of local pseudomeasures and pseudofunctions.

\begin{proposition}
\label{W-Icosistencyprop} Let $G(s)$ be analytic on the half-plane $\Re e \: s > \alpha$ and have local pseudomeasure (local pseudofunction) boundary behavior on $\alpha+iU$. Then, $G$ has pseudomeasure (pseudofunction) boundary behavior on $\Re e \: s = \alpha$ with respect to every $\varphi \in A_{c}(\mathbb{R})$ with $\operatorname*{supp} \varphi \subset U$.
\end{proposition}
\begin{proof} Fix $\varphi \in A_{c}(\mathbb{R})$ with $\operatorname*{supp} \varphi \subset U$.
Let $f\in L^{\infty}(\mathbb{R})$ be such that $\lim_{\sigma\to\alpha^{+}}G(\sigma+it)=\hat{f}(t)$, distributionally, on a neighborhood $V\subset U$ of $\operatorname*{supp} \varphi$. As in the proof of Proposition \ref{W-Iprop1}, one deduces from the edge-of-the-wedge theorem that $G_{1}(s)=G(s)-\mathcal{L}\{f_{+}, s-\alpha\}$ has analytic continuation through $\alpha+iV$, where $f_{+}(x)=f(x) H(x)$. Thus,
\begin{align*}
I_{\varphi}(h)&= \int_{-\infty}^{\infty} G_{1}(\alpha+it)\varphi(t)e^{ih t}\: \mathrm{d}t+ \lim_{\sigma\to0^{+}} \int_{-\infty}^{\infty} \mathcal{L}\{f_{+}; \sigma+it\} \varphi(t) e^{ih t}\mathrm{d}t
\\
&
= o(1)+ \lim_{\sigma\to0^{+}} \int_{0}^{\infty}e^{-\sigma x} f_{+}(x)\hat{\varphi}(x-h) \mathrm{d}x= o(1)+  \int_{-\infty}^{\infty} f(x+h)\hat{\varphi}(x) \mathrm{d}x,
\end{align*}
which is $O(1)$. In the pseudofunction case we may additionally require $\lim_{|x|\to\infty}f(x)=0$, so that $I_{\varphi}(h)=o(1)$.
\end{proof}

Exactly the same argument given in proof of Proposition \ref{thboundedness} would work when pseudomeasure boundary behavior of $\mathcal{L}\{S;s\}$ at $s=1$ is replaced by pseudomeasure boundary behavior on $\Re e\:s=1$ with respect to a single $\varphi\in A_{c}(\mathbb{R})\setminus\{0\}$ with non-negative Fourier transform (which implies $\varphi(0)\neq 0$) if one is able to justify the Parseval relation
$$
 \int^{\infty}_{-\infty} \Delta(x)e^{-\sigma x} \hat{\varphi}(x-h)  \mathrm{d}x=\int_{-\infty}^{\infty}\mathcal{L}\{\Delta; \sigma + it\} e^{iht} \varphi(t) \mathrm{d}t.
$$
But this holds in the $L^{2}$-sense as follows from the next simple lemma\footnote{More precisely, we first apply Lemma \ref{lemma for Parseval} and then modify $S$ in a finite interval so that we may assume that $\Delta(x)e^{-\sigma x}$ belongs to $L^{2}(\mathbb{R})$ for each $\sigma>0$. Clearly, $\varphi\in L^{2}(\mathbb{R})$ as well.}.
\begin{lemma}
\label{lemma for Parseval} Let $S\in L^{1}_{loc}[0,\infty)$ be log-linearly boundedly decreasing with convergent Laplace transform for $\Re e\:s>1$. Then, $S(x)=o(e^{\sigma x})$, $x\to\infty$, for each $\sigma>1$.
\end{lemma}
\begin{proof}
As in the proof of Theorem \ref{thboundedness}, we may assume  that \eqref{eqlsditerated} holds and $S$ is positive. For fixed $\sigma>1$,
\begin{align*}
0<e^{-\sigma h} S(h) &= \frac{\sigma}{1-e^{-\sigma}}\int_{h}^{h+1}(S(h)-S(x))e^{-\sigma x}\mathrm{d}x + o_{\sigma}(1)
\\
&
\leq \frac{\sigma C e^{-(\sigma-1)h}(1-e^{1-\sigma})}{(\sigma-1)(1-e^{-\sigma})} + o_{\sigma}(1)
=o_{\sigma}(1), \quad h\to\infty.
\end{align*}
\end{proof}
The following alternative version of Proposition \ref{thboundedness} should now be clear.

\begin{corollary} \label{thboundedl1} Let $S\in L^{1}_{loc}[0,\infty)$ and let $\varphi\in A_{c}(\mathbb{R})$ be non-identically zero and have non-negative Fourier transform. Then, \eqref{W-Ieq1} holds if and only if $S$ is log-linearly boundedly decreasing, \eqref{LaplaceConvergenceeq} holds, and 
$\mathcal{L}\{S;s\}$ has pseudomeasure boundary behavior on $\Re e\: s=1$ with respect to $\varphi$.
\end{corollary}

Next, we proceed to extend the actual Wiener-Ikehara theorem.

\begin{theorem} \label{thlocalpseudofunction} Let $S\in L^{1}_{loc}[0,\infty)$. Then,
\begin{equation} \label{eqresultth}
 S(x) \sim ae^{x}
\end{equation}
holds if and only if $S$ is log-linearly slowly decreasing, \eqref{LaplaceConvergenceeq} holds, and 
\begin{equation}
\label{LaplacePoleeq}
\mathcal{L}\{S;s\} - \frac{a}{s-1}
\end{equation}
admits local pseudofunction boundary behavior on the whole line $\Re e \: s = 1$.
\end{theorem}

\begin{proof} The direct implication is straightforward. Let us show the converse. We may assume again that $S$ is positive. As before, we set $\Delta(x) = e^{-x}S(x)$. Applying Proposition \ref{thboundedness}, we obtain $\Delta(x) = O(1)$, because $1/(s-1)$ is actually a global pseudomeasure on $\Re e\:s=1$. In particular, we now know that $\Delta\in \mathcal{S}'(\mathbb{R})$. Let $H$ be the Heaviside function. Note that the Laplace transform of $H$ is $1/s$, $\Re e\:s>0$. We then have that the Fourier transform of $\Delta-aH$ is the boundary value of $\mathcal{L}\{S;s+1\}-a/s$ on $\Re e\:s=0$, and thus a local pseudofunction on the whole real line; but this just means that for each $\phi \in \mathcal{F}(\mathcal{D(\mathbb{R})})$
\[
\left\langle \Delta(x) -aH(x), \phi(x-h)\right\rangle =\frac{1}{2\pi}\left\langle \hat{\Delta}(t) -a\hat{H}(t), \hat{\phi}(-t)e^{ith}\right\rangle  =o(1), \quad h\to\infty,
\] 
i.e.,
\begin{equation}\label{eqconv}
\int_{-\infty}^{\infty}\Delta(x+h)\phi(x)\:\mathrm{d}x =a \int_{-\infty}^{\infty}\phi(x)\:\mathrm{d}x+o(1), \quad h\to\infty.
\end{equation} 
Since $\Delta$ is bounded for large arguments, its set of translates $\Delta(x+h)$ is weakly bounded in $\mathcal{S}'(\mathbb{R})$. Also, $\mathcal{F}(\mathcal{D}(\mathbb{R}))$  is dense in $\mathcal{S}(\mathbb{R})$.  We can thus apply the Banach-Steinhaus theorem to conclude that (\ref{eqconv}) remains valid\footnote{In the terminology of \cite{p-s-v}, this means that $\Delta$ has the S-limit $a$ at infinity.} for all $\phi \in \mathcal{S}(\mathbb{R})$. Now, let $\varepsilon > 0$ and choose $\delta$ and $x_{0}$ such that (\ref{defloglsd}) is fulfilled. Pick a non-negative test function $\phi\in\mathcal{D}(0,\delta)$ such that $\int^{\delta}_{0} \phi(x) \mathrm{d}x = 1$. Then,
\begin{align*}
 \Delta(h) & = \int^{\delta}_{0} \Delta(h)\phi(x) \:\mathrm{d}x \leq \varepsilon + \int^{\delta}_{0} e^{x}\Delta(x+h)\phi(x)\: \mathrm{d}x \\
& \leq \varepsilon + e^{\delta}\int^{\delta}_{0} \Delta(x+h) \:\phi(x) \mathrm{d}x = \varepsilon + e^{\delta} (a +o(1)), \quad h\geq x_{0},
\end{align*}          
where we have used \eqref{eqconv}. Taking first the limit superior as $h\to\infty$, and then letting $\delta \rightarrow 0^{+}$ and $\varepsilon \rightarrow 0^{+}$, we obtain $\limsup_{h \rightarrow \infty} \Delta(h)  \leq a$. The reverse inequality with the limit inferior follows from a similar argument, but now choosing the test function $\phi$ with support in $(-\delta,0)$. Hence, (\ref{eqresultth}) has been established.
\end{proof}

We can further generalize Theorem \ref{thlocalpseudofunction} by using the following simple consequence of Wiener's local division lemma. 

\begin{lemma} \label{lemwienerdivision} Let $\phi_{1},\phi_{2} \in L^{1}(\mathbb{R})$ be such that $\operatorname*{supp } \hat{\phi_{2}}$ is compact and that $\hat{\phi}_{1} \neq 0$ on $\operatorname*{supp } \hat{\phi}_{2}$. Let $\tau\in L^{\infty}(\mathbb{R})$ satisfy $(\tau \ast \phi_{1})(h)=o(1)$, then $(\tau \ast \phi_{2})(h) =o(1)$.
\end{lemma}
\begin{proof} By Wiener's division lemma \cite[Chap. II, Thm. 7.3]{korevaarbook}, there is $\psi \in L^{1}(\mathbb{R})$ such that $\hat{\psi} = \hat{\phi_{2}}/\hat{\phi_{1}}$, or $\psi \ast \phi_{1} = \phi_{2}$. Since convolving an $o(1)$-function with an $L^{1}$-function remains $o(1)$, we obtain $(\tau \ast \phi_{2})(h) =( (\tau \ast \phi_{1}) \ast \psi )(h)= o(1)$.
\end{proof}

\begin{theorem}
\label{thol1} Let $S\in L^{1}_{loc}[0,\infty)$ and let $\{\varphi_{\lambda}\}_{\lambda\in J}$ be a family of functions such that $\varphi_{\lambda}\in A_{c}(\mathbb{R})$ for each $\lambda \in J$ and the following property holds:
\begin{enumerate}
 \item [] For any $t \in \mathbb{R}$, there exists some $\lambda_{t}\in J$ such that $\varphi_{\lambda_{t}}(t)\neq 0$. Moreover, when $t=0$, the Fourier transform of the corresponding $\varphi_{\lambda_{0}}$ is non-negative as well.
\end{enumerate}
Then,
$$
 S(x) \sim ae^{x}
$$
 if and only if $S$ is log-linearly slowly decreasing, \eqref{LaplaceConvergenceeq} holds, and the analytic function \eqref{LaplacePoleeq} has pseudofunction boundary behavior on $\Re e \: s = 1$ with respect to every $\varphi_{\lambda}$. 
\end{theorem}

\begin{proof} Once again the direct implication is straightforward, so we only prove the converse. By Corollary \ref{thboundedl1}, it follows that $\Delta(x) := e^{-x}S(x) = O(1)$. Modifying $\Delta$ on a finite interval, we may assume that $\Delta\in L^{\infty}(\mathbb{R})$. The usual calculations done above (cf. the proof of Proposition \ref{thboundedness}) show that $\int_{-\infty}^{\infty} (\Delta(x+h) -aH(x+h))\hat{\varphi}_{\lambda}(x)\mathrm{d}x=o(1)$, $x\to\infty$, for each $\lambda\in J$, where again $H$ denotes the Heaviside function. (We may now apply dominated convergence to interchange limit and integral because $\Delta\in L^{\infty}(\mathbb{R})$.) Pick $t_{0}\in\mathbb{R}$. Lemma \ref{lemwienerdivision} then ensures $\left\langle \hat{\Delta}(t)-a\hat{H}(t),\varphi(t)e^{iht}\right\rangle=\left\langle \Delta(x+h)-aH(x+h),\hat{\varphi}(x)\right\rangle=o(1)$ for all $\varphi\in\mathcal{D}(\mathbb{R})$ with support in a sufficiently small (but fixed) neighborhood of $t_{0}$. This shows that $\hat{\Delta}-a\hat{H}\in PF_{loc}(\mathbb{R})$. Since this distribution is the boundary value of \eqref{LaplacePoleeq}  on $\Re e\:s=1$, Theorem \ref{thlocalpseudofunction} yields $S(x) \sim ae^{x}$.
\end{proof}

Observe that Zhang's theorem (Theorem \ref{thzhang}) follows at once from Theorem \ref{thol1} upon setting $\varphi_{\lambda}(t) = \chi_{[-\lambda,\lambda]}(t) (1-\left|t\right|/\lambda)$. Here one has $\hat{\varphi}_{\lambda}(x) = 4\sin^{2}(\lambda x/2)/(x^{2}\lambda)$. More generally,

\begin{corollary} \label{W-Ic2} Let $S\in L^{1}_{loc}[0,\infty)$ and let $\varphi\in A_{c}(\mathbb{R})$ be non-identically zero such that $\hat{\varphi}$ is non-negative. Then,
$$
S(x)\sim a e^{x}
$$
if and only if  $S$ is log-linearly slowly decreasing, \eqref{LaplaceConvergenceeq} holds, and the analytic function $G(s)=\mathcal{L}\{S;s\}-a/(s-1)$ satisfies:  There is $\lambda_{0}>0$ such that for each $\lambda\geq\lambda_{0}$
\[
 I_{\lambda}(h)=\lim_{\sigma \rightarrow 1^{+}} \int^{\infty}_{-\infty} G(\sigma + it) e^{iht} \varphi\left(\frac{t}{\lambda}\right)\mathrm{d}t
\]
 exists for all sufficiently large $h>h_{\lambda}$ and
$ \displaystyle \lim_{h \rightarrow \infty} I_{\lambda}(h)= 0.
$
\end{corollary}

We conclude the article with two remarks.

\begin{remark}
\label{W-Irk1} Suppose that  $S$ is of local bounded variation on $[0,\infty)$ so that $\mathcal{L}\{S;s\}=s^{-1}\mathcal{L}\{\mathrm{d}S;s\}=s^{-1}\int_{0^{-}}^{\infty}e^{-sx}\mathrm{d}S(x)$. Then, the pseudomeasure boundary behavior of $\mathcal{L}\{S;s\}$ at $s=1$ in Proposition \ref{thboundedness} becomes equivalent to that of $\mathcal{L}\{\mathrm{d}S;s\}$ because the boundary value of $s$ is the invertible smooth function $1+it$ and smooth functions are multipliers for local pseudomeasures (and pseudofunctions). Likewise, the local pseudofunction boundary behavior of \eqref{LaplacePoleeq} in Theorem \ref{thlocalpseudofunction} is equivalent to that of 
\begin{equation}
\label{LaplaceStieltjeseq}
\mathcal{L}\{\mathrm{d}S;s\} - \frac{a}{s-1}.
\end{equation}
On the other hand, we do not know whether the pseudomeasure (pseudofunction) boundary behavior of  $\mathcal{L}\{S;s\}$ (of \eqref{LaplacePoleeq}) with respect to $\varphi$ (with respect to every $\varphi_{\lambda}$) can be replaced by that of $\mathcal{L}\{\mathrm{d}S;s\}$ (of (\ref{LaplaceStieltjeseq})) in Corollary \ref{thboundedl1} (in Theorem \ref{thol1}). The same comment applies to Corollary \ref{W-Ic2}.
\end{remark}
\begin{remark}
\label{W-Irk2} Let $G(s)$ be analytic on the half-plane $\Re e \: s > \alpha$ and suppose it has pseudomeasure (pseudofunction) boundary behavior on $\Re e \: s = \alpha$ with respect to some $\varphi \in A_{c}(\mathbb{R})$. If $\varphi(t_{0})=0$, then $G$ does not necessarily have pseudomeasure (pseudofunction) boundary behavior at $\alpha+it_0$. For example, if $G$ has meromorphic continuation to a neighborhood of $\alpha+it_0$ with a pole of order say $n\geq 2$ at the point $\alpha+it_0$ and if $\varphi$ is such that $\varphi^{(j)}(t_0)=0$ for $j=0,1,\dots,n$ and is supported in a sufficiently small neighborhood of $t_0$, we would have that $\varphi(t)G(\alpha+it)$ is continuous and hence a pseudo\-function, without $G(s)$ having itself pseudomeasure boundary behavior at $\alpha+it_0$. If $\varphi(t_{0})\neq 0$, however, it is unclear to us whether $G$ should have local pseudomeasure (pseudofunction) boundary behavior at $\alpha+it_0$. It would be interesting to establish whether the latter is true or false. Observe this question is closely related to the one raised in Remark \ref{W-Irk1}.

\end{remark}

\end{document}